\documentclass[letterpaper, 10 pt, conference]{ieeeconf}  
\pdfoutput=1
\IEEEoverridecommandlockouts                              
\usepackage{amsmath}

\usepackage{subfiles}
\usepackage{amsthm}
\usepackage{graphicx}
\usepackage{hyperref}
\usepackage{cite}
\theoremstyle{plain}
\newtheorem{theorem}{Theorem}

\theoremstyle{remark}
\newtheorem{remark}{Remark}
\theoremstyle{definition}
\newtheorem{definition}{Definition}

\allowdisplaybreaks

\title{\LARGE \bf
Aerocapture Guidance for Augmented Bank Angle Modulation
}

\author{Kyle A. Sonandres$^{1}$, Thomas Palazzo$^{2}$, and Jonathan P. How$^{3}$
\thanks{*This work was supported by the Charles Stark Draper Laboratory, Inc.}
\thanks{$^{1}$Kyle A. Sonandres is with the Aerospace Controls Lab, Department of Aeronautics and Astronautics, and is a Draper Scholar,
        Massachusetts Institute of Technology, Cambridge, MA 02139, USA
        {\tt\small kyles7@mit.edu}}%
\thanks{$^{2}$Thomas Palazzo is with The Charles Stark Draper Laboratory, Inc., Cambridge, MA 02139, USA
        {\tt\small tpalazzo@draper.com}}%
\thanks{$^{2}$Jonathan P. How is with the Aerospace Controls Lab, Department of Aeronautics and Astronautics, Massachusetts Institute of Technology,
        Cambridge, MA 02139, USA
        {\tt\small jhow@mit.edu}}%
}

\begin{document}

\maketitle
\thispagestyle{empty}
\pagestyle{empty}

\begin{abstract}
This paper presents an optimal control solution for an aerocapture vehicle with two control inputs, bank angle and angle of attack, referred to as augmented bank angle modulation (ABAM). 
We derive the optimal control profiles using Pontryagin's Minimum Principle, 
validate the result numerically using the Gauss pseudospectral method (implemented in GPOPS),
and introduce a novel guidance algorithm, ABAMGuid, for in-flight decision making. High-fidelity Monte Carlo simulations of a Uranus aerocapture mission demonstrate that ABAMGuid can greatly improve capture success rates and reduce the propellant needed for orbital correction following the atmospheric pass.
\end{abstract}

\section{INTRODUCTION}
The National Academy of Science's 2022-2032 Planetary Science Decadal Survey (PSDS) has recommended the Uranus Orbiter and Probe (UOP) as the highest priority flagship-class mission \cite{PSDS}. Because of the distance from Earth, a mission utilizing traditional, fully-propulsive orbital insertion requires a cruise time of 13-15 years and necessitates carrying large amounts of propellant ($\approx$ 1 km/s $\Delta V$) \cite{PSDS}. Aerocapture enables fast arrival velocities, which can reduce cruise time by 2-5 years while requiring less propellant, reducing launch mass and/or leaving more mass for the scientific payload \cite{Restrepo}. For these reasons, NASA has been studying the feasibility of utilizing aerocapture at Uranus \cite{Restrepo, Shellabarger, Scoggins, Matz, Dutta}.

Aerocapture consists of converting a hyperbolic approach trajectory into a target orbit using the aerodynamic drag generated from a single atmospheric pass. Following the atmospheric pass, a periapsis raise and potentially an apoapsis correction is then performed to achieve the target orbit.

The most widely studied aerocapture guidance methodology is bank angle modulation (BAM), in which the bank angle ($\sigma$) is used to rotate the lift vector around the freestream velocity vector ($V_{\infty}$) while keeping the vehicle at a constant angle of attack ($\alpha$), lift ($L$) and drag ($D$). Such control architecture originates from early Apollo-era entry guidance algorithms \cite{moseley1969apollo}, \cite{graves1972apollo}. 
In recent decades, predictor-corrector algorithms have typically been used due to their robustness against large dispersions while remaining computationally tractable for online use \cite{Lu2}. 
PredGuid+A \cite{lafleur2011conditional} is a numerical predictor-corrector (NPC) based on the skip-entry algorithm PredGuid \cite{putnam2008improving} which has been adapted for aerocapture.
More recently, FNPAG \cite{lu2015optimal} is a two-phase NPC algorithm designed to fly the optimal bang-bang bank angle profile, solving for the switching time in phase one and a constant bank angle in phase two to meet the apoapsis targeting constraints.

Recently, there has been a growing interest in direct force control (DFC), an alternative control architecture that uses the angle of attack ($\alpha$) and side-slip angle ($\beta$) to change the lift, drag, and side-force ($Q$). 
An NPC algorithm for DFC is developed in \cite{deshmukh2020investigation} that modifies FNPAG for DFC application. 
A similar approach is deployed in \cite{matz2020development}, but it is found that the analytic solution does not exist without certain assumptions on the aerodynamics.
 They numerically find that the best solution is bang-bang in $\alpha$ and implement a two-phase NPC structure to command $\alpha$ while controlling $\beta$ with a PD controller to drive the orbit into a specified plane. 
Further numerical analysis is done in \cite{geiser2022optimal}, where the authors conclude that a bang-bang angle of attack solution is a good approximation to the optimal solution. 

Recent work has shown that controlling $\alpha$ and $\sigma$ (augmented bank angle modulation, or ABAM) can significantly improve performance at the edges of the entry corridor compared to BAM \cite{palazzo2024}.
ABAM could be an attractive solution to teams looking to minimize deviation from heritage vehicle designs. 
However, a formalization of ABAM based on optimal control theory has yet to be proposed. 
This work presents three main contributions. 
\begin{enumerate}
	\item We derive optimal control profiles for ABAM, showing that the result can be written in terms of two switching curves (the choice of which depends on $\sigma$) that predict when/ how the $\alpha$ switching sequence should occur. This extends the closed form BAM solution used in FNPAG to the two input case. 
	\item We solve the problem numerically using GPOPS, showing consistency between the numerical and theoretical results.
	\item We develop an ABAM aerocapture guidance algorithm (ABAMGuid) based on the optimal control solution. We perform Monte Carlo simulations of UOP aerocapture using ABAMGuid and FNPAG, demonstrating that ABAMGuid results in performance improvements over FNPAG; In conservative entry scenarios, we observe a 29.5\% reduction in 99-th \%-ile post-capture $\Delta V$ and
a 39.3 \% reduction in cases which fail to be captured into the desired orbit relative to FNPAG. 
\end{enumerate} 

\section{OPTIMAL CONTROL PROBLEM}
This section presents a derivation of the optimal control profiles for aerocapture ABAM control. Similar derivations can be found in Ref. \cite{lu2015optimal} for BAM and Ref. \cite{deshmukh2020investigation} for DFC. 
\subsection{Equations of Motion}
The three-dimensional equations of motion of a spacecraft inside the atmosphere of a rotating planet are \cite{miele1989optimal}
\begin{align}
	\dot{r} &= V \sin \gamma, \label{eq: rdotfull} \\
    \dot{\theta} &= \frac{V \cos \gamma \sin \psi}{r \cos \phi}, \\
    \dot{\phi} &= \frac{V \cos \gamma \cos \psi}{r}, \\
    \dot{V} &= -D(\alpha)-g_r \sin \gamma-g_\phi \cos \gamma \cos \psi  \\
    	    &+\Omega^2 r \cos \phi(\sin \gamma \cos \phi-\cos \gamma \sin \phi \cos \psi) , \nonumber \\   	
    	\dot{\gamma} &= \frac{1}{V}\bigl[ L(\alpha) \cos \sigma+\left(V^2 / r-g_r\right) \cos \bigr. \gamma+g_\phi \sin \gamma \cos \psi \\ 
    		&+2 \Omega V \cos \phi \sin \psi \nonumber \\  
    		&\bigl.+\Omega^2  r \cos \phi(\cos \gamma \cos \phi v  +\sin \gamma \cos \psi \sin \phi) \bigr], \nonumber \\
    	\dot{\psi} &= \frac{1}{V}\bigl[ \frac{L(\alpha) \sin \sigma}{\cos \gamma}+\frac{V^2}{r} \cos \gamma \sin \psi \tan \phi+g_\phi \frac{\sin \psi}{\cos \gamma} 		\bigr. \label{eq: psidotfull}\\
    		&-2 \Omega V(\tan \gamma \cos \psi \cos \phi-\sin \phi) \nonumber \\ 
    		& \bigl.+\frac{\Omega^2 r}{\cos \gamma} \sin \psi \sin \phi \cos \phi\bigr], \nonumber
\end{align}
where $r$ is the vehicle's radial distance from the center of the planet, $\theta$ is longitude, $\phi$ is latitude, $V$ is the vehicle's planet-relative velocity, $\gamma$ is the flight path angle of the velocity vector, and $\psi$ is the heading angle of the velocity vector. 
The angular rate vector of the planet self-rotation is $\Omega$, and $g_r$ and $g_\theta$ are the radial and gravitational acceleration components, respectively. 
The lift ($L$) and drag ($D$) accelerations are
\begin{equation} \label{eq: LD}
	L = \frac{1}{2m} \rho V^2 S_{ref} C_L(\alpha), \ D = \frac{1}{2m} \rho V^2 S_{ref} C_D(\alpha).
\end{equation}
Note the dependence of the aerodynamic coefficients on $\alpha$, which differs from most previous aerocapture guidance discussions involving bank angle control in which angle of attack is assumed to be fixed (or prescribed as a function of velocity).
Atmospheric density is defined by $\rho$, $S_{ref}$ is the vehicle's reference area, and $m$ is the vehicle's mass. 
\subsection{Control Variables}
For ABAM, the bank angle, $\sigma$, and angle of attack, $\alpha$, are free to be controlled. As previously noted, $\alpha$ appears in the dynamics indirectly through the aerodynamic coefficients. The bank angle magnitude is constrained as follows
\begin{equation} \label{eq: sigma constraint}
    0 \leq \sigma_{\textrm{min}} \leq \| \sigma \| \leq \sigma_{\textrm{max}} \leq 180,
\end{equation}
where  $\sigma_{\textrm{min}}$ is larger than $0$ degrees (typically $\approx$ 15), and $\sigma_{\textrm{max}}$ is less than $180$ degrees (typically $\approx$ 105-165). These values are set to ensure cross-range control authority. Similarly, the magnitude of angle of attack is bounded as 
\begin{equation} \label{eq: alpha constraint}
    -30 \leq \alpha_{\textrm{min}} \leq \alpha \leq \alpha_{\textrm{max}} \leq 0,
\end{equation}
where $\alpha_{\textrm{min}}$ and $\alpha_{\textrm{max}}$ are negative values within the vehicle's allowable range. Note that we restrict $\alpha$ to be less than zero due to aerodynamic coefficient symmetry around zero and the assumption that the direction of the lift vector will still be controlled by the bank angle. 
\subsection{Constraints and Objectives}
We consider the \textit{apoapsis targeting problem}, where given an initial condition, we wish to find the bank angle and angle of attack control profiles that satisfy the final state constraint
\begin{equation} \label{eq: apoapsis contraint}
    r_a(r_{\textrm{exit}}, V_{\textrm{exit}}, \gamma_{\textrm{exit}}) - r_a^* = 0,
\end{equation}
where $r_{\textrm{exit}}, V_{\textrm{exit}}, \gamma_{\textrm{exit}}$ are the radius, inertial velocity, and inertial flight path angle at atmospheric exit, respectively. $r_a^*$ is the target apoapsis radius, and $r_a$ is defined by keplerian orbital mechanics as
\begin{equation} \label{eq: apoapsis targ}
r_a=a\left(1+\sqrt{1-\frac{V_{\textrm{exit}}^2 r_{\textrm{exit}}^2 \cos ^2\left(\gamma_{\textrm{exit}}\right)}{\mu a}}\right),
\end{equation}
where the semi-major axis, $a$, is
\begin{equation} \label{eq: semimajoraxis}
    a = \frac{\mu}{2\mu / r_{\textrm{exit}} - V_{\textrm{exit}}^2}. \nonumber
\end{equation}
We wish to solve the \textit{optimal aerocapture problem}, which we define as the solution to the \textit{apoapsis targeting problem} which minimizes the propellant ($\Delta V$) required to insert the spacecraft into its final orbit. 
Propellant consumption in the single-burn formulation is
\begin{equation} \label{eq:dv}
\Delta V=\sqrt{2 \mu}\left(\sqrt{\frac{1}{r_a}-\frac{1}{r_a+r_p^*}}-\sqrt{\frac{1}{r_a}-\frac{1}{2 a}}\right),
\end{equation}
where $r_p^*$ is the target periapsis radius and $\mu$ is the planet gravitational parameter.

\subsection{Assumptions and Simplification}
We observe that the objective function and targeting constraint depend only on the terminal values of the longitudinal motion variables ($r, V, \gamma$), which are decoupled from the rest of the equations of motion. 
We ignore small magnitude terms due to planetary rotation and non-spherical gravity, as their effects are minor as long as the lateral motion variables do not appear in a terminal equality constraint \cite{lu2015optimal}. The simplified longitudinal dynamics are  
\begin{align}
    \dot{r} &=V \sin \gamma,  \label{eq: rdot}\\
    \dot{V} &=-D(\alpha)-\frac{\mu \sin \gamma}{r^2}, \label{eq: vdot}\\
    \dot{\gamma} &=\frac{1}{V}\left[L(\alpha) u_1 +\left(V^2-\frac{\mu}{r}\right) \frac{\cos \gamma}{r}\right].\label{eq: gamdot}
\end{align}

Note that in (\ref{eq: gamdot}), $\cos \sigma$ has been replaced with $u_1$ to simplify notation. To ensure an analytical solution is attainable, the aerodynamic coefficients are modeled as linear functions of $\alpha$
\begin{equation} \label{eq: cd fit}
    C_D = C_{D,\alpha} \alpha + C_{D, 0},
\end{equation}
\begin{equation} \label{eq: cl fit}
    C_L = C_{L,\alpha} \alpha + C_{L, 0}.
\end{equation}
The possible values of $\alpha$ are restricted to a range where the maximum $L / D$ error percentage is within $\approx $ 5\%. As shown in \cite{deshmukh2020investigation}, this level of accuracy is valid for blunt bodied aerocapture vehicles. The guidance algorithm presented later will be simulated with the full nonlinear aerodynamics to ensure performance generalizes to the complete model. 

\subsection{Optimal Control Derivation} \label{subsec: OCP derivation}
\begin{definition} \label{def: sigma bam}
Let $\sigma^*_{\textrm{BAM}}$ be the optimal bank angle profile found in \cite{lu2015optimal} for BAM guidance, where $\sigma^*_{\textrm{BAM}} = \sigma_\text{min}$ (lift-up) if $H_\sigma = \lambda_\gamma > 0$, and $\sigma^*_{\textrm{BAM}} = \sigma_\text{max}$  (lift-down) if $H_\sigma < 0$, where $H_\sigma$ is the switching curve associated with $\sigma$. 
\end{definition}
\begin{theorem}
Assume that for the two input case (ABAM), $\sigma^* = \sigma^*_\text{BAM}$, as in Definition \ref{def: sigma bam}. Then during lift-up flight, $\alpha^* = \alpha_{\textrm{min}}$ if $H_{\alpha, \text{up}} > 0$, and $\alpha^* = \alpha_{\textrm{max}}$ if $H_{\alpha, \text{up}} < 0$. During lift-down flight, $\alpha^* = \alpha_{\textrm{min}}$ if $H_{\alpha, \text{down}} > 0$, and $\alpha^* = \alpha_{\textrm{max}}$ if $H_{\alpha, \text{down}} < 0$. Here, $H_{\alpha, \text{up/down}}$ are the switching curves associated with $\alpha$ during lift-up and lift-down flight, respectively. Note that this structure is general, as certain switches do not necessarily have to exist.

\end{theorem}
\begin{proof}
We find the optimal $\alpha$ profile, $\alpha^*$, via forming the Hamiltonian, $H$, and applying Pontryagin's Minimum Principle \cite{pontrjagin1962mathematical} to obtain the optimal control profile, $u^*$ ($u  = \left[ u_1, \ \alpha \right]$), that satisfies the optimality condition, 
\begin{equation}
u^* 
=
\begin{bmatrix}
u_1^* \\
\alpha^*
\end{bmatrix} 
=
\underset{u \in \mathcal{U}}{\arg \min} \ H \left(x, u, \lambda \right),
\end{equation}
where $\mathcal{U}$ is the admissible set of controls. 
The structure of $u_1^*$ is assumed, so the solution will produce $\alpha^*$. There are no running costs, so Hamiltonian is formed by adjoining the costate variables ($\lambda_r, \lambda_v, \lambda_\gamma$) to the dynamics
\begin{equation}
    H = \lambda_r \dot{r} + \lambda_v \dot{v} + \lambda_\gamma \dot{\gamma}.
\end{equation}
Substituting the dynamics from \eqref{eq: rdot}-\eqref{eq: gamdot} yields
\begin{equation} \label{eq: H}
\begin{split}
    H = & \lambda_r \left(V \sin \gamma\right) + \lambda_v \left( -D(\alpha)-\frac{\mu \sin \gamma}{r^2}\right) \\ &+ \lambda_\gamma \left(\frac{1}{V}\left[L(\alpha) u_1 +\left(V^2-\frac{\mu}{r}\right) \frac{\cos \gamma}{r}\right]\right)
\end{split}
\end{equation}
Because $\sigma \in \{ \sigma_{\textrm{min}}, \sigma_{\textrm{max}} \}$, $u_1$ ($u_1 = \cos \sigma$) can take two values. For lift-up flight, $\sigma = \sigma_\textrm{min}$, so $u_1 \approx 1$ since $\sigma_\textrm{min}$ is small. For lift-down flight, $\sigma = \sigma_\textrm{max}$, so $u_1 \approx -1$ since $\sigma_\textrm{max}$ is large. When these values of $u_1$ are substituted into \eqref{eq: H}, there only remains a linear control dependence on $\alpha$ in $H$. We can thus rearrange the Hamiltonian as
\begin{equation}
    H = H_0 + H_{\alpha, \textrm{(up/down)}} \alpha,
\end{equation}
where $H_{\alpha,  \textrm{up}}$ is the switching function associated with $\alpha$ during lift-up flight ($u_1 \approx 1$), $H_{\alpha, \textrm{down}}$ is the switching function associated with $\alpha$ during lift-down flight ($u_1 \approx -1$), and $H_0$ consists of terms that do not depend on the control variable. We can omit $H_0$, as it has no implications for the solution, and substitute in \eqref{eq: LD} for $L(\alpha)$ and $D(\alpha)$ to get 
\begin{align} \label{eq: H nonlinear}
    \tilde{H} &= \left[-\lambda_v \left( q S C_{D, \alpha}\right)  \pm \lambda_\gamma \left( \frac{q}{V} S C_{L,\alpha}\right) \right] \alpha \\
    &= q S \left( -\lambda_v C_{D, \alpha} \pm \lambda_\gamma \frac{C_{L, \alpha}}{V}\right) \alpha \\
    &= H_{\alpha, \textrm{(up/ down)}} \alpha, 
\end{align}
where $q$ is the dynamic pressure. We can drop $q$ and $S$, which do not affect the zero crossing, to get the switching curves
\begin{align}
    H_{\alpha, \textrm{up}}  &= - \lambda_v C_{D, \alpha} + \lambda_{\gamma} \frac{C_{L,\alpha}}{V}, \label{eq: H11}\\
    H_{\alpha, \textrm{down}}  &= - \lambda_v C_{D, \alpha} - \lambda_{\gamma} \frac{C_{L,\alpha}}{V}. \label{eq: H12}
\end{align}
Because $\tilde{H}$ is linear in $\alpha$, it can be minimized by selecting $\alpha_\textrm{max}$ if the switching function is negative, and $\alpha_\textrm{min}$ if the switching function is positive. The result is a bang-bang $\alpha$, where in lift-up flight the optimal control is 
\begin{equation} \label{eq: aoa1}
    \alpha^*= \begin{cases}\alpha_{\text {min }}, & \text { if } H_{\alpha,\textrm{up}}>0 \\ \alpha_{\text {max }}, & \text { if } H_{\alpha,\textrm{up}}<0\end{cases}, 
\end{equation}
and in lift-down flight the optimal control is 
\begin{equation} \label{eq: aoa2}
    \alpha^*= \begin{cases}\alpha_{\text {min }}, & \text { if } H_{\alpha,\textrm{down}}>0 \\ \alpha_{\text {max }}, & \text { if } H_{\alpha,\textrm{down}}<0\end{cases}.
\end{equation}
\end{proof}
\vspace{-2mm}
\begin{remark}
The control in the case where $ H_{\alpha,\textrm{up}}$ and $ H_{\alpha,\textrm{down}}$ equal zero for a non-zero time interval represents a singular arc, and it has been shown to be impossible for similar problems \cite{lu2015optimal}, \cite{deshmukh2020investigation}, but was not investigated in this work. 
\end{remark}
\begin{remark}
Given the assumption that the optimal $\sigma$ profile follows the results in \cite{lu2015optimal}, we expect to see
\begin{equation} \label{eq: sig}
    \sigma^*= \begin{cases}\sigma_{\text {min }}, & \text { if } \lambda_\gamma >0 \\ \sigma_{\text {max }}, & \text { if } \lambda_\gamma<0 \end{cases}
\end{equation}
Notice that this solution is recovered at the intersection of the lift-up and lift-down switching functions, where $H_{\alpha,\textrm{up}} = H_{\alpha,\textrm{down}}$. We can therefore define the bank angle switching ($H_\sigma$) function accordingly, dropping the constant multiplication factor
\begin{equation}
    H_\sigma = H_{\alpha,\textrm{up}} - H_{\alpha,\textrm{down}} = \lambda_\gamma
\end{equation}
\end{remark}
\begin{remark} \label{remark 3}
We assume that the control profile consists of at most three switches ($\sigma$ may switch once and $\alpha$ may switch once both before and after the $\sigma$ switch). The following numerical analysis supports this assumption. Considering a nominal scenario where the vehicle begins flying at $\sigma_{\textrm{min}}$ and $\alpha_{\textrm{min}}$, the first switching time occurs when $H_{\alpha, \textrm{up}}$ crosses zero and corresponds to an $\alpha$ switch. The second occurs when $H_{\alpha, \textrm{up}} = H_{\alpha, \textrm{down}}$ and corresponds to a $\sigma$ switch. The third occurs when $H_{\alpha, \textrm{down}}$ crosses zero and corresponds to the second $\alpha$ switch. 

While we assume the vehicle begins flying $\sigma_{\textrm{min}}$ and $\alpha_\textrm{min}$ for consistency with FNPAG, it remains possible for other permutations of control sequences to exist. 
For example, a full-lift down flight is possible, and may be thought of as a special case of the bang-bang control scheme, where the duration of the $\sigma = \sigma_\textrm{min}$ phase is zero. In this case, $H_{\alpha, \textrm{down}}$ governs a potential $\alpha$ switch from $\alpha_\textrm{min}$ to $\alpha_\textrm{max}$. However, it is also possible for the duration of the $\alpha = \alpha_\textrm{max}$ phase to be zero. This would result in a trajectory where no control variable switches occur. 
For these reasons, the three-switch control profile is not an inherently limiting architecture.  
\end{remark}

\subsection{Numerical Analysis using GPOPS}
To lend validity to the theory presented above, the optimal control problem was solved numerically using the Gauss Pseudospectral Optimization Software (GPOPS) \cite{benson2005gauss, rao2010algorithm}. 
When encoding the problem, no switching times are enforced, so any similarities between the numerical results and control theory have arisen organically.  

The longitudinal dynamics used are as in \eqref{eq: rdot}-\eqref{eq: gamdot}. We enforce a final state constraint according to the apoapsis targeting criteria in \eqref{eq: apoapsis contraint}. For convergence improvements, a new control variable is introduced $u = \left[ u_1, \quad u_2 \right] = \left[ \cos \sigma, \quad \sin \sigma \right]$. Note that this introduces a path constraint $u_1^2 + u_2^2 = 1$.
The constraint in (\ref{eq: sigma constraint}) therefore becomes $u_{1, \textrm{min}} \leq u_1 \leq u_{1, \textrm{max}}$, where $u_{1,\textrm{max}} = \cos \sigma_{\textrm{min}}$ and $u_{1, \textrm{min}} = \cos \sigma_{\textrm{max}}$.
The cost to be minimized is $\Delta V$, as in (\ref{eq:dv}). Linear aerodynamic models are implemented according to (\ref{eq: cd fit}) and (\ref{eq: cl fit}). 
Initial guesses are obtained by propagating dynamics given an initial state, using a constant $\alpha$ profile, and a bang-bang $\sigma$ profile. 
The atmospheric density profile is constructed via polynomial fit to the UranusGRAM nominal density profile. 

We use GPOPS to find the optimal solution in nominal, shallow, and steep entry conditions. Once the solution is found, the state and estimated costate information are extracted to plot the switching functions in \eqref{eq: H11}-\eqref{eq: H12} alongside the optimal $\alpha$ and $\sigma $ profiles, shown in Fig. \ref{Fig: GPOPS joint fig}. 

Let us first consider the \textit{nominal EFPA} case. The vehicle begins flying lift-up, with $u_1 \approx 1$. Therefore, we expect $H_{\alpha, \textrm{up}}$ to govern the first $\alpha$ switch. 
$H_{\alpha, \textrm{up}}$ goes from positive to negative just before $t \approx 180$, and $\alpha$ goes from $\alpha_{\textrm{min}}$ to $\alpha_{\textrm{max}}$ at about the same time. 
The next predicted event is a $\sigma$ switch at the point where $H_{\alpha, \textrm{up}} = H_{\alpha, \textrm{down}}$. This occurs at $t \approx 226$, and we see that $\sigma$ switches from $\sigma_{\textrm{min}}$ to $\sigma_{\textrm{max}}$ ($u_1 \approx -1$) at $t \approx 226$. 
Because the vehicle is now flying lift-down, $u_1 \approx -1$, and we expect a second $\alpha$ switch when $H_{\alpha, \textrm{down}}
$ crosses zero. 
This occurs at $t \approx 253$, and the corresponding $\alpha$ switch happens near that time. 

In the \textit{shallow EFPA} case, the vehicle begins flying lift-down. Thus, we expect $H_{\alpha, \textrm{down}}$ to govern a potential $\alpha$ switch. $H_{\alpha, \textrm{down}}$ goes from positive to negative at $t \approx 449$, and the corresponding $\alpha$ switch from $\alpha_\textrm{min}$ to $\alpha_\textrm{max}$ occurs at that time. $H_{\alpha, \textrm{up}}$ and $H_{\alpha, \textrm{down}}$ next begin to converge on each other, but do not cross, suggesting that no $\sigma$ switch occurs. As observed in the optimal $\sigma$ profile, there are no switches, thus $H_{\alpha, \textrm{up}}$ never becomes active and a second $\alpha$ switch does not occur. 

In the \textit{steep EFPA} case, we begin lift up, so $H_{\alpha, \textrm{up}}$ is active. $H_{\alpha, \textrm{up}}$ does not cross zero before intersecting with $H_{\alpha, \textrm{down}}$, which signifies a bank angle switch at $t \approx 296$. We see that $\sigma$ switches from $\sigma_{\textrm{min}}$ to $\sigma_\textrm{max}$ at this time. No $\alpha$ switch occurs, which is expected since the $H_{\alpha, \textrm{up}}$ did not cross zero. With the $\sigma$ switch to $\sigma_\textrm{max}$, $H_{\alpha, \textrm{down}}$ becomes active. While $H_{\alpha, \textrm{down}}$ does cross zero, this crossing occurs before the function is active. Thus, we expect no $\alpha$ switch, which is what we observe.

\begin{figure*}[ht!]
\centering
\includegraphics[width=\textwidth]{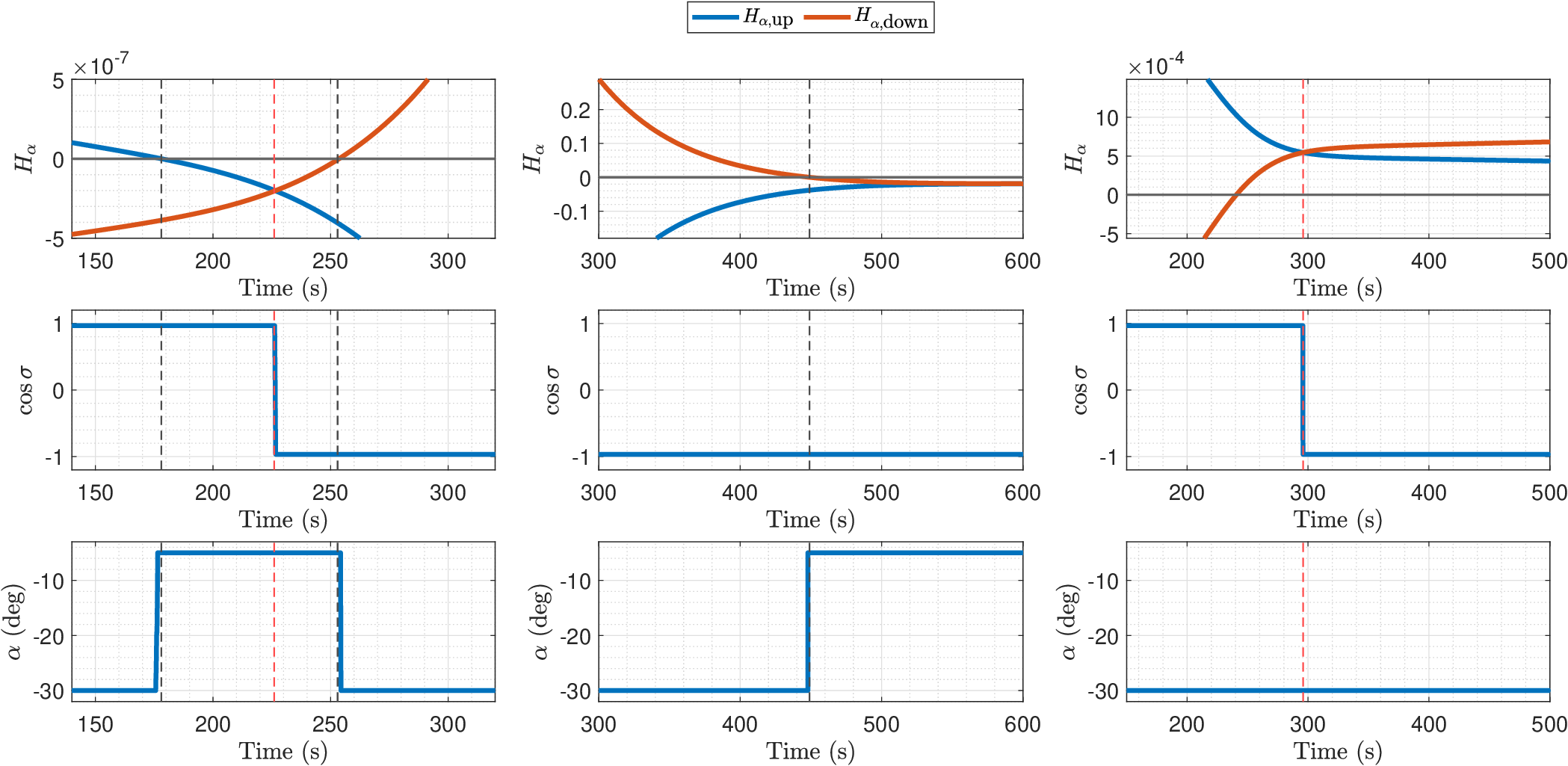} 
\vspace{-7mm} \caption{Switching functions (top row), optimal bank angle (center row) and angle of attack (bottom row) profiles. Switching functions plotted from GPOPS state and estimated costate output, controls plotted from GPOPS control solution. Left column corresponds to the nominal entry flight path angle (EFPA) case, center column corresponds to a shallow EFPA case, and right column corresponds to a steep EFPA case. Black vertical dashed lines correspond to $\alpha$ switching times and red vertical dashed lines correspond to $\sigma$ switching times.} \label{Fig: GPOPS joint fig}. 
\end{figure*}

Despite no enforcement of switching times in the optimal control software, the numerical results are consistent with the control theory presented in the previous section. This suggests that the theory is plausibly valid under our set of assumptions. The following section will present a novel algorithm that aims to exploit the bang-bang solution structure for use in an ABAM guidance algorithm. 

\section{AEROCAPTURE GUIDANCE ALGORITHM}
This section presents an ABAM aerocapture guidance algorithm (ABAMGuid) based on the analysis in the previous section. We seek to enable closed-loop ABAM guidance via numerically solving for the three discrete switching times in-the-loop of guidance using a combination of numerical and heuristic optimization methods. Note that this guidance algorithm is designed to fly the optimal solution profile previously identified, but is not a true globally optimal algorithm, which would require solving the constrained nonlinear control problem in real time at each guidance call. Instead, our solution reduces the problem complexity to a three variable solve, which is much more computationally efficient and will be shown to perform well in practice.  

\subsection{A Four-Phase Algorithm} 
\textbf{Phase 1} is defined from entry interface to the first switching time, $t_{s,1}$. This phase is activated via a g-load trigger (set to 0.1 g's), indicating that the vehicle has enough control authority to begin active guidance. During this phase, the predictor uses the bang-bang solutions discovered in the previous section, parameterized by three switching time variables, while the Nelder-Mead method \cite{nelder1965simplex} solves for the switching times which minimize the apoapsis targeting error. That is, we seek to minimize
\begin{equation} \label{Eq: apoapsis error}
    f(t_{s,1}, t_{s,2}, t_{s,3}) = \frac{1}{2} \left( r_a - r_a^*\right)^2.
\end{equation}
A zeroth-order method was selected to avoid invertibility issues that can arise when inverting a Jacobian, as is required in the Gauss-Newton method used in \cite{Lu2}. The simplex $S$ is initialized by adding a constant value to each basis given an initial best guess $(t_{s,1}^0, t_{s,2}^0, t_{s,3}^0)$, which results in $n$ = 4 points in the simplex. Temporal consistency is enforced by penalizing violations of $t_{s,1} \leq t_{s,2} \leq t_{s,3}$ with large function values. Iteration is stopped once the stopping criteria is met
\begin{equation}
    \text{std}(S_f)< \epsilon_{NM},
\end{equation}
where $S_f$ are the function values of each point in the simplex and $\epsilon_{NM}$ is a preselected satisfactorily small value. During phase 1, guidance commands $\sigma_{\textrm{min}}$ and $\alpha_{\textrm{min}}$.

\textbf{Phase 2} is defined from $t_{s,1}$ to $t_{s,2}$ and is triggered when $t > t_{s,1}$. This phase utilizes a similar solve as phase 1, using Nelder-Mead to minimize (\ref{Eq: apoapsis error}). Because $t > t_{s,1}$, the search is over two variables $t_{s,2}$ and $t_{s,3}$, and the simplex size is $n$ = 3. During phase 2, guidance commands $\sigma_{\textrm{min}}$ and $\alpha_{\textrm{max}}$.

\textbf{Phase 3} is defined from $t_{s,2}$ to $t_{s,3}$ and is triggered when $t > t_{s,2}$. In phase 3, we seek to minimize (\ref{Eq: apoapsis error}). Since $t_{s,1}$ and $t_{s,2}$ are in the past, the problem becomes a univariate optimization problem and the apoapsis error is minimized by solving for $t_{s,3}$ using the Newton-Rhapson method and a secant scheme derivative approximation using the two previous iterations. This update is

\begin{equation}
    t_{s,3}^{(k+1)} = t_{s,3}^{(k)} - \frac{z(t_{s,3}^{(k)})}{\left[ z(t_{s,3}^{(k)}) - z(t_{s,3}^{(k-1)}) \right]} \left( t_{s,3}^{(k)} - t_{s,3}^{(k-1)}\right).
\end{equation}
Iteration is stopped when the stopping criteria is met
\begin{equation}
    \| z(t_{s,3}^{(k+1)}) \frac{\partial z(t_{s,3}^{(k+1)})}{\partial t_{s,3}}  \| \leq \epsilon_{NR},
\end{equation}
where $\epsilon_{NR}$ is a preselected small value and $z$ is defined as the apoapsis targeting error $r_a - r_a^*$. Further discussion on this process can be found in \cite{Lu2}. During phase 3, guidance commands $\sigma_{\textrm{max}}$ and $\alpha_{\textrm{min}}$.

\textbf{Phase 4} is defined from $t_{s,3}$ to atmospheric exit $t_f$ and is triggered when $t > t_{s,3}$. Identical to phase two of FNPAG, we seek to find a constant magnitude bank angle that minimizes the apoapsis targeting error using Brent's method \cite{brent2013algorithms}. Angle of attack is held constant at $\alpha_{\textrm{min}}$. The constant bank angle found via Brent's method will not be equal to $\sigma_{\textrm{max}}$ due to trajectory dispersion, but will typically have large magnitude in near nominal conditions. 

\subsection{Lateral Guidance, Constraints, and Uncertainty}
Lateral guidance is not considered in this work, but discussion on a separate logic channel to utilize $\sigma$ reversals for lateral targeting can be found in \cite{lu2015optimal} and \cite{Lu2}. Similarly, heating and load constraints can be implemented via a predictive approach developed in \cite{Lu2} but are not considered here. 
Modeling uncertainties in vehicle parameters and atmospheric density can degrade performance due to inaccurate prediction of aerodynamic accelerations.
To mitigate this, we deploy the filtering techniques explained in \cite{lu2015optimal} to estimate the deviation of sensed acceleration from expected. 
\vspace{-1mm}
\section{SIMULATION RESULTS}
This section presents closed-loop Uranus Orbiter and Probe (UOP) aerocapture simulation results using the proposed algorithm, ABAMGuid, and FNPAG. FNPAG was implemented following the derivation in \cite{lu2015optimal}.
\subsection{Simulation Methodology}
Aerocapture trajectories are simulated using a high-fidelity 3-DoF simulation developed in MATLAB/ Simulink. The full 3-DoF dynamics from (\ref{eq: rdotfull})-(\ref{eq: psidotfull}) are implemented along with the complete nonlinear aerodynamics model. A Monte Carlo (MC) set of 8,000 simulations is performed with entry and target conditions, vehicle parameter dispersions, and simulation settings summarized in Table \ref{table: MC dispersions}. 
\begin{table}[t]
\caption{Monte Carlo Simulation dispersions and parameters}
\vspace{-6mm}
\label{table: MC dispersions}
\begin{center}
\begin{tabular}{|c||c|}
\hline
Parameter & Value\\
\hline
EI Altitude, $h_0$ (km) & 1,000 \\
\hline
Inertial Velocity, $V_0$ (km/s) &  23.78\\
\hline
Entry Flight Path Angle, $\gamma_0$ (deg) & -10.8 $\pm$ 0.189 3-$\sigma$ (Baseline), \\ &-10.8 $\pm$ 0.622 3-$\sigma$ (Conservative)\\
\hline
Entry Longitude, $\theta$ (deg) & 262.12 \\
\hline
Entry Latitude, $\phi$ (deg)& -16.02 \\
\hline
Entry Azimuth Angle, $\psi_0$ (deg) & 117.45 \\
\hline
Mass, $m_0$ (kg) & 4,063\\
\hline
Target Apoapsis Altitude (km) & 2,000,000\\
\hline
Target Periapsis Altitude (km) & 4,000\\
\hline
Bank Angle Limits (deg) & 15 $\leq \| \sigma \| \leq $ 165 \\
\hline
Angle of Attack Limits (deg) & -25 $\leq \alpha \leq $ -10 \\
\hline
\end{tabular}
\end{center}
\end{table}
For each simulation, the true density profile is varied by sampling the UranusGRAM atmosphere model. 
The on-board atmosphere model is created by fitting a polynomial to the natural log of the nominal GRAM density profile. 
\subsection{Results}
Four MC sets are presented. FNPAG is compared to the new guidance algorithm, ABAMGuid, given a baseline and conservative set of entry states. The baseline set is intended to be a realistic representation of potential mission uncertainties and is denoted with (B) in Table \ref{table: DV Stats} and Table \ref{table: failure breakdown}. The conservative set is intended to be more challenging and is used to stress-test the algorithms by increasing the 3-$\sigma$ delivery uncertainties \cite{mages2025mission}, denoted with (C). The $\Delta V$ statistics are summarized in Table \ref{table: DV Stats}. Cases where the probe fails to be captured are omitted from $\Delta V$ statistics. 

Note that mission success is defined as having complete post-aerocapture operations. 
Failed simulations are further categorized as \textit{landers} in cases where the resulting orbital period is less than 10 days, and \textit{hyperbolic} in cases where the orbital period is greater than 2.5 years. A breakdown of failures and failure modes is included in Table \ref{table: failure breakdown}.
\begin{table}[t]
\caption{Monte Carlo $\Delta V$ statistics}\
\label{table: DV Stats}
\begin{center}
\vspace{-7mm}
\begin{tabular}{|c||c||c||c|}
\hline
Case & $\Delta V$ mean & $\Delta V$ 3-$\sigma$ & $\Delta V$ 99th $\%$-ile\\
\hline
FNPAG (B) & 35.2 & 43.0 & 81.9 \\
\hline
\textbf{ABAMGuid} (B) & \textbf{32.8} & \textbf{37.6} & \textbf{74.1} \\
\hline
FNPAG (C) & 43.0 & 83.4 & 167.0  \\ 
\hline
\textbf{ABAMGuid} (C) & \textbf{35.7} & \textbf{58.8} & \textbf{119.9} \\
\hline
\end{tabular}
\end{center}
\end{table} 
In the baseline entry set, the proposed algorithm results in a 16.7\% improvement in 99th percentile $\Delta V$. In the conservative entry set, the proposed algorithm results in a 29.5\% improvement in 99th percentile $\Delta V$ and 825 fewer failed simulations (39.3\% reduction), demonstrating that the new method's biggest improvement is performance in stressful conditions. 
Individual simulation results from the conservative MC sets are visualized in Fig. \ref{Fig: mc failures} with a black dot for passed runs, and a red circle for failed runs.
The entry flight path angle (EFPA) of each run is plotted on the x-axis, along with the resulting peak g-load (top row) on the y-axis. The bottom figure is a 1-D representation of simulation outcomes by algorithm to enable direct comparison of entry corridor width. 
We observe a clear increase in the width spanned by the black lines compared to red, suggesting that ABAMGuid enables mission success over a wider entry corridor than FNPAG. 
\begin{figure}[t!]
\centering
\includegraphics[width=0.49\textwidth]{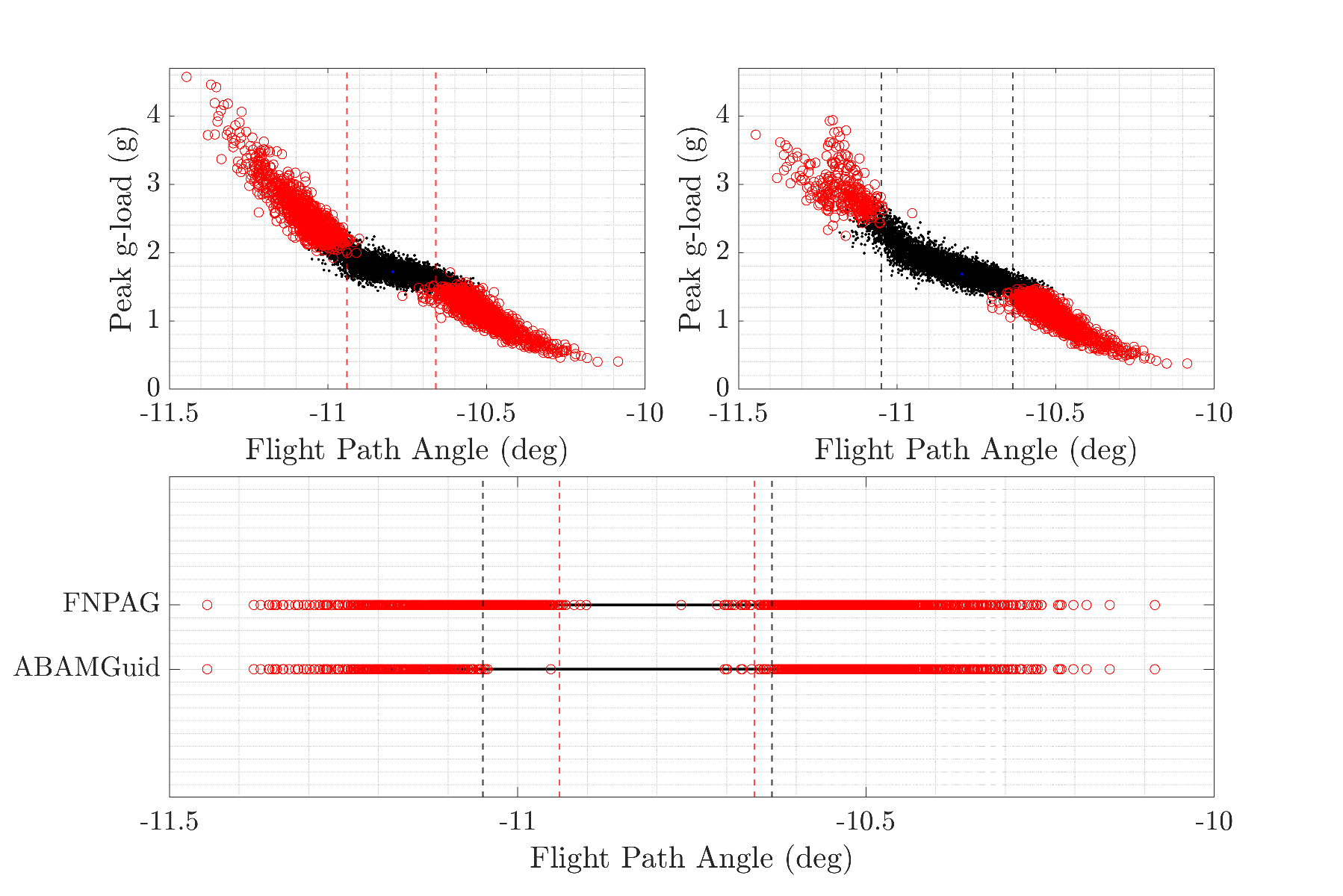} 
\vspace{-3mm}
\caption{MC simulation peak g-load (top row) and outcome (bottom) results as functions of entry flight path angle with conservative entry states. Black points indicate passing runs and red circles indicate failures. The vertical dashed lines are placed such that there is a $\geq$ 95\% success rate at a given EFPA between the lines, with red corresponding to FNPAG, and black to ABAMGuid. Top right corresponds to ABAMGuid, top left to FNPAG.} \label{Fig: mc failures}
\end{figure}
\begin{table}[t]
\caption{Breakdown of Monte Carlo simulation failures}
\label{table: failure breakdown}
\begin{center}
\vspace{-4mm}
\begin{tabular}{|c||c||c||c|}
\hline
Case & Failures (Pass \%) & Lander & Hyperbolic \\
\hline
FNPAG (B) & 6 (99.93) & 0 & 6  \\
\hline
\textbf{ABAMGuid} (B) & \textbf{1 (99.99)} &  0 & \textbf{1} \\
\hline
FNPAG (C) & 2098 (73.78) & 324 & 1774 \\ 
\hline
\textbf{ABAMGuid} (C) & \textbf{1273 (84.09)} & \textbf{101} & \textbf{1172}  \\
\hline
\end{tabular}
\end{center}
\end{table}

\section{CONCLUSIONS}
We present ABAMGuid, a novel aerocapture guidance algorithm for ABAM inspired by the derivation of a three-switch bang-bang optimal control solution.
We validated the optimal solution in GPOPS and rigorously tested ABAMGuid through high-fidelity Monte Carlo analysis. 
As the presented algorithm is intended to be run online, run time and convergence may be critical to mission success. While simulation run-time was not observed to be significantly impacted in comparison to FNPAG, future work could include a formal assessment of the computational implications of solving the multi-variable optimization problem in algorithm phases 1 and 2. 
Further, the Nelder-Mead method does not enjoy the convergence guarantees of Newton-Rhapson. 
In our implementation, care was taken to initialize a good initial simplex, but future work could be done to identify if and where breakdowns occur in this process. 
Lastly, future work should assess the affects of incorporating lateral guidance and heating/ load constraints. 

\addtolength{\textheight}{-12cm}   


\bibliographystyle{IEEEtran}
\vspace{-3mm}\bibliography{references}

\end{document}